\documentclass[11pt]{amsart}
\usepackage{amsthm}
\usepackage{amssymb}
\usepackage{bbm}
\usepackage{eucal}
\usepackage{url}
\usepackage{dsfont}
\usepackage{upgreek}
\usepackage{mathrsfs}
\usepackage[all, arc]{xy}
\usepackage[perpage]{footmisc}
\usepackage[toc, page]{appendix}
\usepackage{enumerate}
\usepackage{lettrine}
\usepackage[lite, numeric,sorted,traditional-quotes,compressed-cites]{amsrefs}

\newcommand{\id}[1]{{\text{id}}_{#1}}

\newcommand{\Fil}[3]{{\rm Fil}^{#1}_{#2}(#3)}
\newcommand{\Filt}[2]{{\rm Fil}^{#1}_{#2}}
\newcommand{\Grr}[3]{{\rm gr}^{#1}_{#2}(#3)}

\newcommand{\Z}{{\mathbbm Z}}
\newcommand{\Q}{{\mathbbm Q}}

\newcommand{\spec}[1]{{\text{Spec}(#1)}}

\newcommand{\sproj}[2]{{\mathbbm P}^{#1}_{#2}}

\newcommand{\cali}[1]{{\mathscr #1}}
\newcommand{\supp}[1]{{\rm Supp}(#1)}
\newcommand{\codim}[2]{{\rm codim}(#1,#2)}

\newcommand{\dime}[1]{{\rm dim}(#1)}

\newcommand{\class}[2]{{\rm cl}_{#1}({#2})}
\renewcommand{\id}[1]{{\rm id}_{#1}}

\newcommand{\tor}[4]{{{\rm Tor}}^{#1}_{#2}(#3,#4)}

{\newtheoremstyle{astatement}
{13pt}
{13pt}
{\it}
{}
{\bf}
{.$-$}
{.5em}
{}

\theoremstyle{astatement}

}
\newtheoremstyle{adefinition}
{13pt}
{13pt}
{}
{}
{\scshape}
{.}
{.5em}
{}

\theoremstyle{adefinition}

\newtheoremstyle{statement}
{13pt}
{13pt}
{\it}
{}
{\bf}
{.$-$}
{.5em}
{}

\theoremstyle{statement}

\newtheorem{theorem}{Theorem}[section]
\newtheorem*{theorem-introduction}{Theorem}
\newtheorem{lemma}[theorem]{Lemma}
\newtheorem{proposition}[theorem]{Proposition}

\newtheorem{cor}[theorem]{Corollary}

\newtheoremstyle{definition}
{13pt}
{13pt}
{}
{}
{\scshape}
{.}
{.5em}
{}

\theoremstyle{definition}

\newtheorem{example}[theorem]{Example}
\newtheorem{remark}[theorem]{Remark}

\newtheoremstyle{definitionprime}
{13pt}
{13pt}
{}
{}
{\scshape}
{$'$.}
{.5em}
{}

\theoremstyle{definitionprime}

\newtheoremstyle{remarks}
{13pt}
{13pt}
{}
{}
{\scshape}
{.}
{.5em}
{}

\theoremstyle{remarks}

\newtheoremstyle{remarksb}
{13pt}
{13pt}
{}
{}
{\scshape}
{.}
{\newline}
{}

\theoremstyle{remarksb}

\newtheoremstyle{underlined}
{13pt}
{13pt}
{\sl}
{}
{\scshape}
{}
{.5em}
{}

\theoremstyle{underlined}

\makeatletter 

\def\section@cntformat{\S\thesection.\ }
\def\subsection@seccntformat{\S\thesubsection \ }

\makeatother
\numberwithin{equation}{theorem}
\begin{document}
\title{On coniveau filtration of Grothendieck group of a scheme}


\author{Shahram Biglari}
\address{School of Mathematics, Institute for studies in fundamental sciences (IPM), P.O.~Box: 19395-5746, Tehran, Iran}
\curraddr{} \email{biglari@ipm.ir}
\thanks{}




\date{2012}
\maketitle
\section{Introduction}\label{introduction}
In this work we deal with the coniveau filtration on the Grothendieck group $K_0$ of coherent modules on a scheme and its behavior under presence of other algebraic structures on $K_0$.

For a (noetherian) scheme $S$ of finite dimension $d$, the Grothendieck group of coherent modules on $S$, denoted by $K_0'(S)$, has a decreasing filtration
\[
K_0'(S)=:\Filt{0}{\rm top}\supseteq \Filt{1}{\rm top}\supseteq\dotsc \supseteq\Filt{d+1}{\rm top}=0
\]
where the subgroup $\Filt{q}{\rm top}$ is generated by classes of modules whose support is of codimension at least $q$ (see \S 2). Now assume that $S$ is in addition regular (and separated). The group $K_0'(S)$ is then isomorphic to Grothendieck group $K_0(S)$ of locally free coherent modules (cf.~\ref{pro:theta_X-iso}). In particular $K_0(S)$ is a ring whose underlying additive group is equipped with a decreasing filtration (induced from the isomorphism $K_0\simeq K_0'$). The important question is whether this filtration is multiplicative; i.e. if
\[
\Filt{p}{\rm top}\cdot \Filt{q}{\rm top}\subseteq\Filt{p+q}{\rm top}.
\]
A result of Grothendieck appearing in [SGA 6, Exp. 0] (and using Chow's moving lemma) gives a positive answer to the question in the case of smooth quasi-projective schemes over a field.

Here we drop some assumptions on the base scheme and prove a first case of a general multiplicativity result. More precisely, we prove:
\begin{theorem-introduction}[\ref{thm:multiplicativity-regular}]
If $Z_1$ and $Z_2$ are regularly embedded closed subschemes of respective codimension $q_1$ and $q_2$ in a regular noetherian separated scheme $S$, then
\[
\class{}{\cali{O}_{Z_1}}\cdot \class{}{\cali{O}_{Z_2}}\in \Fil{q_1+q_2}{\rm top}{S}.
\]
\end{theorem-introduction}
The ring $K_0(S)$ has a $\lambda-$ring structure (see the discussion preceding~\ref{lem:K0-lambda-functor}). There naturally arises the question of whether this structure is in an appropriate sense compatible with the coniveau filtration. This is therefore within the context of the vast discourse of Riemann-Roch (\'a la Grothendieck). 

Classically, an up to torsion isomorphism between the $\gamma-$filtration and the coniveau filtration is established first (cf.~\ref{thm:gamma-top-iso-Q} where a proof of this is presented). Then it will be clear that these filtrations on $K_0(S)\otimes_\Z\Q$ commute with such natural operations as $\lambda^n$ and $\gamma^n$.  

We prove that the Adams operations $\psi_n$ (cf.~\eqref{def:adams}) respect the coniveau filtration on $K_0(S)$ in a strong sense. More precisely:
\begin{theorem-introduction}[\ref{thm:adams-top-fil}]
If $S$ is a regular noetherian separated scheme, then for any $n\geq 1$, $q\geq 1$ and $x\in \Fil{q}{\rm top}{S}$ we have
\[
\psi_n(x)-n^{q}x\in \Filt{q+1}{\rm top}.
\] 
\end{theorem-introduction}
\section{Grothendieck groups and their filtrations}
We first recall a few definitions and facts from [SGA 6, Exp. IV, \S 2] on Grothendieck groups. For a (noetherian) scheme $X$ let
\[
K_0'(X):=K_0\bigl(D^b_{\rm coh}(X)\bigr)
\]
be the Grothendieck group of the derived category of (cohomologically) bounded complexes of ${\mathscr O}_X-$modules with coherent cohomology modules. This is an abelian group and (for noetherian $X$) can also be defined as the (\emph{na\"{i}ve}) Grothendieck group of the abelian category of coherent modules.

For an arbitrary scheme $S$, we also consider the Grothendieck group $K_0(S)$ of the exact category of locally free coherent $\cali{O}_S-$modules. Note that the latter is in fact the \emph{na\"{i}ve} $K_0$ of the scheme. The better definition is considered to be the Grothendieck group
\[
K^\circ(S):=K_0\bigl(D(S)_{\rm parf}\bigr) 
\]
of the triangulated category of perfect complexes. For a scheme $S$ there are natural homomorphisms arising from categorical embeddings;
\[
K_0(S)\xrightarrow{i} K^{\circ} (X)\xrightarrow{\theta} K_0'(S).
\]
If $S$ is a noetherian scheme having an ample family of invertible modules in the sense of [SGA 6, Exp. II, 2.2.4] (resp. being separated regular), then $i$ (resp. $\uptheta$) is an isomorphism.
\begin{remark} The above constructions are functorial. More precisely, the following assertions (all explained in [SGA 6, Exp. IV, 11-12]) hold:
\begin{enumerate}[{\rm \qquad 1.}]
 \item $S\mapsto K_0(S)$ (resp. $S\mapsto K^\circ (S)$) defines a contravariant functor from schemes to commutative rings.
 \item $S\mapsto K^\circ (S)$ defines a covariant functor from noetherian schemes with proper morphisms of finite tor-dimension to abelian groups.
 \item $X\mapsto K_0'(S)$ defines a covariant functor from noetherian schemes with proper morphisms to abelian groups and a contravariant functors from (noetherian) schemes with morphisms of finite tor-dimension to abelian groups.
\end{enumerate}
\end{remark}
The tensor product (resp. derived tensor product) of modules (resp. complexes) gives $K_0(S)$ (resp. $K^\circ (S)$) the structures of a commutative ring so that $i$ is a homomorphism of rings. There is similarly obtained a pairing
\begin{equation}\label{eq:K0G0-pairing}
K^\circ (X)\otimes_\Z K_0'(X)\to K_0'(X),\quad s\otimes x\mapsto s\cdot x.
\end{equation}
\begin{proposition}[Projection formula]\label{thm:projection-formula}
Let $f\colon X\to S$ be a proper morphism of noetherian schemes. The formula
\[
f_\ast (f^\ast (s)\cdot x)=s\cdot f_\ast (x)
\]
holds for all $s\in K^\circ(S)$ and $x\in K_0'(X)$. If $f$ is in addition perfect, the formula holds in each of the following cases:
\begin{enumerate}[{\rm \qquad 1.}]
\item for all $s\in K^\circ (S)$ and $x\in K^\circ(X)$,
\item for all $s\in K_0'(S)$ and $x\in K^\circ(X)$.
\end{enumerate}
\end{proposition}
\begin{proof}
These are explained in [SGA 6, Exp. IV, 2.11-12] and all come down to an isomorphism
\[
Rf_\ast (\cali{E}\otimes_X^{\mathbbm L} \cali{F})\simeq Rf_\ast (\cali{E})\otimes_S^{\mathbbm L} \cali{F}
\]
for perfect complexes $\cali{F}$ and coherent complexes $\cali{E}$ (treated in full details in [SGA 6, Exp. III, 3.7]).
\end{proof}
\begin{remark}
We also note that with assumptions as in the main statement above, the projection formula holds when $s\in K_0(S)$. This is because the $K_0(S)-$module structure of $K_0'(S)$ is defined through the homomorphism $i$ and the $K^\circ (S)-$module structure. 
\end{remark}

Now we consider the gamma (or Grothendieck) filtration of the ring $K_0$ of a scheme. A reference for the basic results is [SGA 6, Exp. V, VI].

For a scheme $S$, the Grothendieck group of the category of locally free $S-$modules of bounded finite rank is denoted by $K_0(S)$. As mentioned earlier, using the tensor product of modules, this becomes a ring. Defining for each $n\geq 0$ and each $x=\class{}{\cali{E}}$ the element
\[
\lambda^n(x):={\rm cl}\bigl({\rm Alt}^{n}_{\cali{O}_S}(\cali{E}) \bigr),
\]
gives $K_0(S)$ a (pre-)$\lambda-$ring structure: $\lambda^0(x)=1$, $\lambda^1(x)=x$ and the map
\[
\lambda=\lambda_t\colon K_0(S)\to 1+tK_0(S)[\![t]\!],\quad x\mapsto \sum \lambda^n(x)t^n
\]
is a homomorphism of groups (cf.~[SGA 6, Exp. V, \S 2]). Assume that $S$ is connected. In this case there is an augmentation $\epsilon\colon K_0(S)\to \Z$ given by $\class{}{\cali{E}}\mapsto {\rm rk}(\cali{E})$. Note that $\epsilon$ is compatible with each $\lambda^i$; for $x=\class{}{\cali{E}}$ we have from the corresponding results for affine schemes:
\[
\epsilon(\lambda^n(x))=\lambda^n(\epsilon(x))=(n!)^{-1}\epsilon(x)\cdot (\epsilon(x)-1)\dotsc (\epsilon(x)-n+1).
\]
\begin{lemma}\label{lem:K0-lambda-functor}
The assignment $S\mapsto K_0(S)$ defines a contravariant functor from {\emph{(connected)}} schemes to the category of \emph{($\Z$-augmented)} $\lambda-$ rings.
\end{lemma}
\begin{proof}
The assertion that the above pre-$\lambda-$ring structure is a $\lambda-$ring structure means that there is a ring structure on $\Lambda \bigl(K_0(S)\bigr):=1+tK_0(S)[\![t]\!]$ (with its addition being induced from the multiplicative group $K_0(S)[\![t]\!]^\times$) and a pre-$\lambda-$ring structure on it such that $\lambda_t$ defined above is a ring homomorphism compatible with $\lambda^i$ (for details and a proof see [SGA 6, Exp. VI, 3.3]). The functoriality follows from the fact that for any morphism $f\colon X\to S$, there is a canonical isomorphism
\[
f^\ast{\rm Alt}^{n}_{\cali{O}_S}(\cali{E})\simeq {\rm Alt}^{n}_{\cali{O}_X}(f^\ast\cali{E}).\qedhere
\]
\end{proof}
For $s\in K_0(S)$ set
\begin{equation}\label{eq:gamma-series}
\gamma_t(s):=\sum_{i\geq 0}\gamma^i(s)t^i:=\lambda_{t/{1-t}}(s).
\end{equation}
The $\gamma-$filtration of $K_0(S)$ is defined as follows. For each $p\geq 0$ define ${\rm Fil}^q_\gamma(S)\subseteq K_0(S)$ to be the subgroup generated by all expressions $\gamma^{i_1}(x_1)\cdot \gamma^{i_2}(x_2)\cdot\dotsc\cdot\gamma^{i_p}(x_p)$ with $\epsilon(x_{j})=0$ and $\sum i_j\geq q$. It is convenient to define ${\rm Fil}^n_\gamma(S)=K_0(S)$ for each $n\leq 0$. By definition
\[
{\rm Fil}^1_\gamma(S)={\rm ker}(\epsilon).
\]
\begin{proposition}\label{pro:gamma-filtration} The following assertions hold.
\begin{enumerate}
\item Each ${\rm Fil}^q_\gamma(S)$ is an ideal of $K_0(S)$,
\item for $p,q\in\Z$ we have
\[
{\rm Fil}^p_\gamma(S)\cdot {\rm Fil}^q_\gamma(S)\subseteq {\rm Fil}^{p+q}_\gamma(S).
\]
\end{enumerate}
\end{proposition}
\begin{proof}
These are clear from the definition.
\end{proof}
Now consider the group $K_0'(X)$ defined for each noetherian scheme $X$. Through the pairing~\eqref{eq:K0G0-pairing} we shall always consider $K_0'(X)$ as a module over $K_0(X)$. Recall the definition of coniveau (or topological) filtration on $K_0'(X)$: for each $j\geq 0$, the subgroup ${\rm Fil}_{\rm top}^jK_0'(X)\subseteq K_0'(X)$ is defined to be generated by classes of coherenet modules having a support of codimention $\geq j$. Note that by definition
\[
\Fil{i}{\rm top}{S}=0\quad{\rm for\ all\ }i>\dime{S}.
\]
\begin{proposition}\label{lem:k0-cycle-length}
Let $\cali{E}$ be a coherent module over a noetherian scheme $X$. If $\codim{\supp{\cali{E}}}{X}\geq q$, then
\[
\class{}{\cali{E}}-\sum_{x\in X^{(q)}}{\rm lg}_{\cali{O}_{X,x}}(\cali{E}_x)\class{}{\cali{O}_{\overline{\{x\}}}}\in \Fil{q+1}{\rm top}{X}
\]
where $X^{(q)}\subseteq X$ is the set of generic points of integral closed subschemes of codimension $q$ in $X$.
\end{proposition}
\begin{proof}
This is proved in [SGA 6, Exp. X, 1.2] where the definition and results on the topological filtration is given in terms of dimension rather than codimension. We re-present the proof written for the filtration above. For this set $X'=\supp{\cali{E}}$. Let $K'$ be the full subcategory of the category of coherent modules over $X$ formed by modules $\cali{E}'$ with $\supp{\cali{E}'}\subseteq X'$ and for which the assertion holds. Additivity of the length function implies that $K'$ is a serre subcategory. Also $0\in {K'}$ and for each integral closed subscheme $Z\subseteq X'$, the coherent ${X}-$module $\cali{O}_{Z}$ is in ${K'}$. The statement follows from the d\'evissage [EGA III$_1$, 3.1.2].
\end{proof}
\begin{lemma}\label{lem:top-filtration-functor}
Let $f\colon S\to T$ be a morphism of noetherian schemes.
\begin{enumerate}
 \item If $f$ is flat, then $f^\ast \Fil{q}{\rm top}{T}\subseteq \Fil{q}{\rm top}{S}$.
\item If $f$ is proper and surjective, then $f_\ast \Fil{q}{\rm top}{S}\subseteq \Fil{q-d}{\rm top}{T}$ where $d\in Z$ is any integer with
\[
d\geq {\rm dim}({f^{-1}(f(s))})\quad {\rm for\ all\ }s\in S.
\]
\end{enumerate}
\end{lemma}
\begin{proof}For these we will use the result [EGA IV$_2$, 5.5.2] on the dimension: let $t=f(s)$, then
\[
{\rm dim}(\cali{O}_{S,s})\leq {\rm dim}(\cali{O}_{T,t})+{\rm dim}(\cali{O}_{S,s}\otimes_{\cali{O}_{T,t}}k(t)).
\]
For (1): Note that $\supp{f^\ast\cali{E}}=f^{-1}\supp{\cali{E}}$ for a (coherent) $T-$module $\cali{E}$. Since $f$ is flat, the above becomes an equality and hence for each closed subset $Z$ of $T$ we have
\[
\codim{f^{-1}Z}{S}\geq \codim{Z}{T}.
\]
To prove (2) note that for each coherent $S-$module $\cali{E}$ it follows (from e.g. the flat base change) that $\supp{R^nf_\ast\cali{E}}\subseteq f\supp{\cali{E}}$ for all $n\geq 0$. On the other hand for each closed subset $Z\subseteq S$, from the above inequality for dimensions we conclude that
\[
\codim{Z}{S}\leq \codim{fZ}{T}+d
\]
where $d$ is larger than dimension of $\cali{O}_{S,s}\otimes_{\cali{O}_{T,t}}k(t)$ where $t=f(s)$ for any $s\in Z$.
\end{proof}

\begin{example}\label{ex:G0-localization-pullpack}
For $s\in S$, the canonical morphism $h\colon \spec{\cali{O}_{S,s}}\to S$ is flat. The induced (surjective) homomorphism from~\ref{lem:top-filtration-functor} is
\[
h^\ast\colon K_0'(S)\twoheadrightarrow K_0'(\cali{O}_{S,s}),\quad \class{}{\cali{E}}\mapsto \class{}{\cali{E}_s}.
\]
Let $\cali{E}$ be a coherent $S-$module. It follows from the proof of~\ref{lem:top-filtration-functor} that for any $s\in S$ we have
\begin{equation}\label{eq:G0-localization-pullpack}
\codim{\supp{\cali{E}_s}}{\cali{O}_{S,s}}\geq \codim{\supp{\cali{E}}}{S}.
\end{equation}
In the particular case where $s$ is a maximal point, the homomorphism $h^\ast$ is given by $\class{}{\cali{E}}\mapsto {\rm lg}_{\cali{O}_{S,s}}(\cali{E}_s)$. It follows that if $\class{}{\cali{E}}\in \Filt{1}{\rm top}$, then $\codim{\supp{\cali{E}}}{S}\geq 1$.  
\end{example}
\begin{proposition}\label{thm:G0-localization}
For any closed subscheme $i\colon Z\to S$ of a noetherian scheme $S$ with the open complement $j\colon U\to S$, the sequence
\[
K_0'(Z)\xrightarrow{i_\ast} K_0'(S)\xrightarrow{j^\ast} K_0'(U)\to 0
\]
is exact.
\end{proposition}
\begin{proof}
A proof can be found in [SGA 6, Exp. IX, 1.1].
\end{proof}
\begin{lemma}\label{lem:j*-surjective}
Let the notations be as in~\ref{thm:G0-localization}. For each $q$, the homomorphism $j^\ast\colon \Fil{q}{\rm top}{S}\to \Fil{q}{\rm top}{U}$ is surjective.
\end{lemma}
\begin{proof}
Let $\cali{E}$ be a coherent $U-$module with $\codim{\supp{\cali{E}}}{U}\geq q$. There is an injective homomorphism $\cali{E}\to j^\ast j_\ast \cali{E}$ of $\cali{O}_U-$modules. Therefore there is an $\cali{O}_S-$submodule $\cali{G}$ of $j_\ast(\cali{E})$ with $j^\ast (\cali{G})\simeq \cali{E}$. Note that \[\supp{\cali{G}}\subseteq \supp{j_\ast(\cali{E})}\subseteq (\text{the closure of ${\supp{\cali{E}}}$ in $S$}).\] 
But the codimension of the latter is exactly $\codim{\supp{\cali{E}}}{U}$.
\end{proof}

As mentioned earlier $K_0'(S)$ is considered as a module over the ring $K_0(S)$. Considering both of these objects with their filtrations we note that $K_0'(S)$ is a filtered modules over the filtered ring $K_0(S)$, that is:
\begin{proposition}\label{thm:gamma-top-filtration}
The following assertions hold.
\begin{enumerate}[{\rm \qquad 1.}]
\item Each $\Fil{q}{\rm top}{S}$ is a $K_0(S)-$submodule of $K_0'(S)$.
\item For all integers $p,q\in\Z$ we have
\[
\Fil{p}{\gamma}{S}\cdot \Fil{q}{\rm top}{S}\subseteq \Fil{q+p}{\rm top}{S}.
\]
\end{enumerate}
\end{proposition}
\begin{proof}
The first assertion follows from the second or alternatively note that if $\cali{E}$ (resp. $\cali{F}$) is a coherent (resp. locally free) $S-$module, then
\[
\supp{\cali{F}\otimes_S\cali{E}}=\supp{\cali{E}}
\]
and hence $K_0(S)\cdot \Fil{q}{\rm top}{S}\subseteq \Fil{q}{\rm top}{S}$. The second assertion is proved in [SGA 6, Exp. X, 1.3]. We give a slightly different proof. We need to show that for $\cali{E}$ and $\cali{F}$ as above with $\codim{\supp{\cali{E}}}{X}\geq q$ and $\epsilon (\cali{F})=n$ we have
\[
\gamma^p(\class{}{\cali{F}}-n)\cdot\class{}{\cali{E}}\in \Fil{q+p}{\rm top}{S}.
\]
For $p=0$ there is nothing to prove and for $p=1$ this follows directly from~\ref{lem:k0-cycle-length}. For the general case we use induction on $n$ to prove the above for all $p,q$, and $S$. The case $n=1$ follows from the case $p=1$. Let $f\colon S'\to S$ be the projective fiber of $\cali{F}$. There is a rank one module $\cali{L}'$ on $S'$ with
\[
f^\ast \class{}{\cali{F}}-n=(\class{}{\cali{F}'}-n')+(\class{}{\cali{L}'}-1).
\]
where $n'=n-1$. Use the case $p=1$ applied in $S'$ and obtain
\[
s':=[H]^{n'}\cdot f^\ast\class{}{\cali{E}}\in \Fil{q+n'}{\rm top}{S'}
\]
where $[H]=1-\class{}{\cali{O}_{S'}(-1)}\in K_0(S')$. By induction and the case $p=1$ we conclude that $s'':=\gamma^p(f^\ast\class{}{\cali{F}}-n)\cdot s'$ is in $\Fil{q+n'+p}{\rm top}{S'}$. It is enough to note that $\gamma^p(\class{}{\cali{F}}-n)\cdot\class{}{\cali{E}}=f_\ast (s'')$.
\end{proof}
There is a natural homomorphism
\[
\uptheta_X\colon K_0(S)\to K_0'(S),\quad s\mapsto s\cdot \class{}{\cali{O}_S}.
\]
\begin{cor}\label{gamma-subset-top}
For each $p\geq 0$ we have
\[\theta_X \Filt{p}{\gamma}(S)\subseteq \Filt{p}{\rm top}(S).\]
\end{cor}
\begin{proof}
This follows from~\ref{thm:gamma-top-filtration} with $q=0$.
\end{proof}

\begin{proposition}\label{pro:theta_X-iso}
Assume that $S$ is a regular noetherian separated scheme. The following assertions hold:
\begin{enumerate}[{\rm \qquad 1.}]
\item $\uptheta_S$ is an isomorphism.
\item $\Filt{j}{\rm top}(S)=\theta_S \Filt{j}{\gamma}(S)$ for $j=0,1$.
\end{enumerate}
\end{proposition}
\begin{proof}
The first assertion can be found in [SGA 6, Exp. IV] (and follows from the fact that $S$ has an ample family of invertible modules in the sense of [SGA 6, II, 2.2.4]). For the second note that $j=0$ is already the surjectivity of $\uptheta_S$. For $j=1$, let $\cali{E}$ be a coherent module the codimension of whose support is $\geq 1$. Assume that $S$ is connected. There is a complex $\cali{L}^\bullet$ of locally free modules and a quasi-isomorphism $\xi\colon \cali{E}\to \cali{L}^\bullet$. Let $s\in S\setminus \supp{\cali{E}}$. It follows that $\cali{L}_s^\bullet$ is quasi-isomorphic to zero. By~\ref{lem:K0-lambda-functor} the map $K_0(S)\to K_0(\cali{O}_{S,s})$ commutes with the augmentation and hence $\epsilon (\class{}{\cali{E}})=\epsilon (0)=0$. Therefore $\class{}{\cali{E}}=\class{}{\cali{L}^\bullet}\in {\rm ker}(\epsilon)$. In this case the class of $\cali{L}^\bullet$ defines an element $l\in \Fil{1}{\gamma}{S}$. By definition $\uptheta_S (l)=\class{}{\cali{E}}$.
\end{proof}
\begin{cor}\label{thm2:fd+1=0}
If $S$ is a regular noetherian separated scheme, then $\Fil{q}{\gamma}{S}=0$ for all $q>\dime{S}$.
\end{cor}
\begin{proof}
Note that by~\ref{pro:theta_X-iso} the map $\theta_S\colon K_0(S)\to K_0'(S)$ is an isomorphism. Using~\ref{gamma-subset-top} the result follows from the fact that by definition $\Fil{q}{\rm top}{S}$ vanishes for all $q>\dime{S}$. 
\end{proof}

\section{A multiplicativity result}
Let $S$ be a regular noetherian separated scheme. We may use the result~\ref{pro:theta_X-iso} to define a ring structure on $K_0'(S)$. Equivalently, the underlying abelian group of the ring $K_0(S)$ has a (topological) filtration $\Fil{q}{\rm top}{S}:=\uptheta_S^{-1}\Fil{q}{\rm top}{S}$. As mentioned in the introduction, the question of multiplicativity of coniveau (i.e. topological) filtration on $K_0(S)$ asks if
\[
\Filt{p}{\rm top}\cdot \Filt{q}{\rm top}\subseteq \Filt{p+q}{\rm top}.
\]
\begin{proposition}\label{pro:multiplicativity-top-fil}
Let the notations be as above. 
\begin{enumerate}[{\rm (i)}]
 \item Assume that $\dime{S}<\infty$. The following assertions are equivalent.
\begin{enumerate}[{\qquad \rm 1.}]  
\item $\Filt{p}{\rm top}(S)\cdot \Filt{q}{\rm top}(S)\subseteq \Filt{q+p}{\rm top}(S)$ for all $p,q\in \Z$.
\item For any integral closed subscheme $Z\hookrightarrow S$ of codimension $\geq p$ and any integer $q\in\Z$ we have
 \[
  \class{}{\cali{O}_Z}\cdot \Fil{q}{\rm top}{S}\subseteq \Fil{q+p}{\rm top}{S}.
 \] 
\item For any integral closed subschemes $Z\hookrightarrow S$ and $Y\hookrightarrow S$ of respective codimensions $\geq p$ and $\geq q$ we have
 \[
  \class{}{\cali{O}_Z}\cdot \class{}{\cali{O}_Y}\in \Fil{q+p}{\rm top}{S}.
 \] 
\end{enumerate}
\item $\Filt{1}{\rm top}(S)\cdot \Filt{q}{\rm top}(S)\subseteq \Filt{q+p}{\rm top}(S)$ for all $q\in \Z$.
\end{enumerate}
\end{proposition}
\begin{proof} (i): ${\rm 1}\Rightarrow{\rm 2}:$ We just need to note that $\class{}{\cali{O}_Z}\in \Fil{p}{\rm top}{S}$. ${\rm 2}\Rightarrow{\rm 3}:$ This is evident. ${\rm 3}\Rightarrow{\rm 1}:$ Using~\ref{lem:k0-cycle-length} and a decreasing induction on $p+q$, it is enough to show that $a\cdot b\in \Filt{p+q}{\rm top}$ where $a=\class{}{\cali{O}_Y}$ (resp. $b=\class{}{\cali{O}_Z}$) for an integral closed subscheme $Y$ (resp. $Z$) of codimension $p$ (resp. $q$). But this is exactly our assumption. (ii): Note that $\Filt{1}{\rm top}(S)=\Filt{1}{\gamma}(S)$ by~\ref{pro:theta_X-iso} and hence the result follows from~\ref{thm:gamma-top-filtration}.
\end{proof}
\begin{lemma}\label{lem:regular-codim}
Let $Y\hookrightarrow X$ be a regular closed immersion of noetherian schemes and $Z\subseteq Y$ a closed subscheme. If $Y$ is pure codimensional, then
\[
\codim{Z}{X}=\codim{Z}{Y}+\codim{Y}{X}.
\]
\end{lemma}
\begin{proof}
We consider this as well-known.
\end{proof}
\begin{proposition}\label{thm:cod-1-inverse-image-top-filt}
Let $i\colon Z\hookrightarrow S$ be a regular closed immersion of pure codimension $1$ of regular noetherian separated schemes. Then
\[
i^\ast (\Fil{q}{\rm top}{S})\subseteq \Fil{q}{\rm top}{Z}
\]
for each $q\geq 0$.
\end{proposition}
\begin{proof}
Fix $q\geq 0$. Let $\cali{E}$ be a coherent $S-$module with $\codim{\supp{\cali{E}}}{S}\geq q$. We show $i^\ast\class{}{\cali{E}}\in \Filt{q+p}{\rm top}$. Using~\ref{lem:regular-codim}, it follows that
\begin{align*}
 \codim{\supp{i^\ast\cali{E}}}{Z} & =\codim{\supp{\cali{E}}\cap Z}{S}-1\geq q-1.
\end{align*}
We may assume that the equality holds. Using~\ref{lem:k0-cycle-length} for all cohomology modules of the complex $Li^\ast (\cali{E})$, we obtain
\[
\class{}{Li^\ast\cali{E}}-\sum_{z\in Z^{(q-1)}}{\rm lg}_{\cali{O}_{Z,z}}(Li^\ast(\cali{E})_z)\class{}{\cali{O}_{\overline{\{z\}}}}\in \Filt{q}{\rm top}.
\]
Let $W=\spec{A}$ be an affine open subscheme of $S$ containing $z\in Z^{(q-1)}$. Let $N$ and $M$ be the $A-$modules corresponding to (the restriction to $W$ of) $i_\ast\cali{O}_Z$ and $\cali{E}$ respectively. Let the prime idea ${\mathfrak p}\subseteq A$ be corresponding to $z$. We note that ${\rm lg}_{\cali{O}_{Z,z}}(Li^\ast(\cali{E})_z)=\chi^A_{\mathfrak p}(M,N)$ where
\[
\chi^A_{\mathfrak p}(M,N)=\sum_{n\geq 0} (-1)^n{\rm lg}_{A_{\mathfrak p}}\bigl(\tor{A_{\mathfrak p}}{n}{M_{\mathfrak p}}{N_{\mathfrak p}} \bigr).
\]
For codimensions in $\spec{A_{\mathfrak p}}$ we have
\begin{align*}
\codim{\supp{M_{\mathfrak p}\otimes N_{\mathfrak p}}}{A_{\mathfrak p}} & = \dime{A_{\mathfrak p}}\\ & =q \\ & <q+1\\
& \leq \codim{M_{\mathfrak p}}{A_{\mathfrak p}}+\codim{N_{\mathfrak p}}{A_{\mathfrak p}}
\end{align*}
where in the last inequality we have used~\ref{ex:G0-localization-pullpack}. Therefore by the vanishing theorem (part of Serre's multiplicity conjectures) proved in~\cite[5.6]{gillet-soule-1987} and~\cite{roberts-1987} we obtain $\chi^A_{\mathfrak p}(M,N)=0$. This implies that $i^\ast \class{}{\cali{E}}=\class{}{Li^\ast\cali{E}}\in \Filt{q}{\rm top}$.	
\end{proof}
\begin{theorem}\label{thm:multiplicativity-regular}
If $Y$ and $Z$ are regularly embedded closed subschemes of respective pure codimensions $p$ and $q$ in a regular noetherian separated scheme $S$, then
\[
\class{}{\cali{O}_Z}\cdot \class{}{\cali{O}_Y}\in \Fil{p+q}{\rm top}{S}.
\]
\end{theorem}
\begin{proof}
We assume $p\geq q$. Denote $\cali{O}_Y$ (considered as an $S-$module) by $\cali{E}$. Consider the blow-up of $S$ along $Z$;
\[
\xymatrix{
Z'\ar[d]^-{f'} \ar@{^(->}[r]^-{i'} & S'\ar[d]^-{f}\\
Z\ar@{^(->}[r]^-{i} & S.
}
\]
The precise definition is given in [EGA II, 8.1.3]. The morphism $f$ is a surjective perfect projective morphism. Therefore $Rf_\ast$ and $Lf^\ast$ give well-defined homomorphisms on $K_0'$ (cf.~[SGA 6, Exp. IV, 2.12 \& Exp. VII, 1.9]). Also note that the morphism $f'\colon Z'\to Z$ is isomorphic to $\sproj{}{}(\cali{N}_{Z/S})\to Z$ where $\cali{N}_{Z/S}$ is the co-normal sheaf of $i$ which is a locally free $Z-$module of rank $p$ (cf.~[EGA IV$_4$, 19.4]). The morphism $i'$ being a regular closed immersion is a perfect morphism and hence defines a homomorphism $i'_\ast\colon K_0(Z')\to K_0(S')$. Let $j$ (resp. $j'$) be the open complement of $i$ (resp. $i'$). Form the diagram 
\[
\xymatrix{
\Fil{q-1}{\rm top}{Z'}'\ar[r]^-{i'_\ast} & \Fil{q}{\rm top}{S'}\ar[r]^-{{j'}^\ast} & \Fil{q}{\rm top}{S'\setminus Z'}\to 0\\
 & \Fil{q}{\rm top}{S}\ar[r]^-{{j}^\ast}\ar[ur]^-{(f\circ j')^\ast} & \Fil{q}{\rm top}{S\setminus Z}\ar@<12px>[u]_-{{f''}^\ast}
}
\]
in which $\Fil{q-1}{\rm top}{Z'}'$ is the subgroup of $K_0'(Z')$ consisting of elements $z'$ with $i'_\ast(z')\in \Filt{q}{\rm top}$ and $f''$ is the restriction of $f$ to $S'\setminus Z'$. The diagram is commutative with an exact row. In fact $j'^\ast$ is by~\ref{lem:j*-surjective} surjective. If $s'\in \Fil{q}{\rm top}{S'}$ such that $j'^\ast(s')=0$, then it follows from~\ref{thm:G0-localization} that $s'=i'_\ast (z')$ for some $z'\in K_0'(Z')$. This shows that the diagram above has an exact row. Consider in particular the element $f^\ast (\class{}{\cali{E}})\in K_0'(S')$. Using the diagram we obtain
\begin{equation}\label{eq:f*E-decomposition}
f^\ast (\class{}{\cali{E}})=e+i'_\ast(z') 
\end{equation}
for some $e\in \Fil{q}{\rm top}{S'}$ and some $z'\in K_0'(Z')$. Next we claim that the codimension (in $S'$) of each closed subset $X'\subseteq S'$ is $\leq \codim{fX'}{S}$. To see this let $x=f(x')\in fX'$ with $\dime{\cali{O}_{S,x}}$ being the codimension of $fX'$ in $S$. If $x\not\in Z$, then $\dime{\cali{O}_{S,x}}=\dime{\cali{O}_{S',x'}}$ which is $\geq \codim{X'}{S'}$. Otherwise assume that $f^{-1}(x)\subseteq Z'$. It follows from the dimension formula for $f'$ that
\[
\dime{\cali{O}_{Z',x'}}=\dime{\cali{O}_{Z,x}}+d_x
\]
where $d_x$ is dimension of the local ring of $f^{-1}(x)$ at $x'$. By regularity of the embeddings $i$ and $i'$ we conclude that
\begin{align*}
\dime{\cali{O}_{S',x'}}&=\dime{\cali{O}_{Z',x'}}+1\\
&= \dime{\cali{O}_{Z,x}}+d_x+1\\
&= \dime{\cali{O}_{S,x}}+(d_x+1-p)
\\& \leq \dime{\cali{O}_{S,x}}.
\end{align*}
This shows that $\codim{X'}{S'}\leq \codim{fX'}{S}$. In particular, since for each coherent modules $\cali{E}$ and each $i\geq 0$ the support of ${R^if_\ast (\cali{E})}$ is $\subseteq f\supp{\cali{E}}$, it follows that
\[
f_\ast\Fil{n}{\rm top}{S'}\subseteq \Fil{n}{\rm top}{S}
\]
for all $n\geq 0$. Consider the element $\lambda_{-1}(\cali{F})\in K_0(Z')$ where
\[
0\to \cali{F}\to {f'}^\ast\cali{N}_{Z/S}\to \cali{O}_{Z'}(1)\to 0
\]
is the exact sequence associated to the canonical epimorphism ${f'}^\ast\cali{N}_{Z/S}\to \cali{O}_{Z'}(1)$. It follows from the projection formula~\ref{thm:projection-formula} and the the identity $f_\ast(\lambda_{-1}(\cali{F}))=1$ from [SGA 6, Exp. VI] that 
\[
f_\ast \bigl(i'_\ast\lambda_{-1}(\cali{F})\cdot f^\ast(\cali{E})\bigr)=i_\ast(1_Z)\cdot \class{}{\cali{E}}.
\]
We shall compute the left hand side using the equation~\eqref{eq:f*E-decomposition}: for the second term of the resulting equation we use [SGA 6] to obtain
\[
i'_\ast (\lambda_{-1}(\cali{F}))\cdot i'_\ast(z')=i'_\ast (\lambda_{-1}(\cali{F})\cdot (1-\class{}{\cali{O}_{Z'}(1)})\cdot z').
\]
Since $\cali{O}_{Z'}(1)$ is of rank one, it follows from the exact sequence above that $\lambda_{-1}(\cali{F})\cdot (1-\class{}{\cali{O}_{Z'}(1)})=\lambda_{-1}({f'}^{\ast}\cali{N}_{Z/S})$. The later is by definition equal to $(-1)^{p}\gamma^{p}(f'^\ast\cali{N}_{Z/S}-p)$ which belongs to $\Filt{p}{\gamma}$. Therefore by~\ref{thm:gamma-top-filtration} and~\ref{lem:top-filtration-functor} we obtain
\begin{align*}
f_\ast \bigl(i'_\ast\lambda_{-1}(\cali{F}))\cdot i'_\ast(z')\bigr) & = i_\ast\bigl(\lambda_{-1}(\cali{N}_{Z/S})\cdot {f'}_\ast (z'))\in \Filt{p+q}{\rm top}.
\end{align*}
For the first term we use the projection formula,~\ref{thm:cod-1-inverse-image-top-filt}, the fact that $\lambda_{-1}(\cali{F})\in \Filt{p-1}{\gamma}$ and~\ref{thm:gamma-top-filtration} to obtain
\begin{align*}
 i'_\ast(\lambda_{-1}(\cali{F}))\cdot e & =  i'_\ast\bigl(\lambda_{-1}(\cali{F})\cdot {i'}^\ast(e)\bigr)\in \Filt{p+q}{\rm top}
\end{align*}
This in view of the compatibility of $f_\ast$ with the topological filtration proven above implies that $f_\ast\bigl(i'_\ast(\lambda_{-1}(\cali{F}))\cdot e\bigr)$ belongs to $\Filt{p+q}{\rm top}$. Puting these together, we have
\[
i_\ast(1_Z)\cdot \class{}{\cali{E}}\in \Filt{p+q}{\rm top}.
\]
The result follows.
\end{proof}
\begin{remark} The proof shows that for given $p\geq q$ and any regular immersion $i\colon Z\to S$ of codimension $p$ we have
 \[
\class{}{\cali{O}_Z}\cdot \Fil{q}{\rm top}{S}\subseteq \Fil{p+q}{\rm top}{S}. 
 \]
\end{remark}
\begin{remark}
For $S$ as above denote the quotient $\Fil{m}{\gamma}{S}/\Fil{m+1}{\gamma}{S}$ by $\Grr{m}{\gamma}{S}$. Define the graded group
\[
\Grr{\bullet}{\gamma}{S}:=\coprod_{m\geq 0} \Grr{m}{\gamma}{S}.
\]
This is a commutative graded algebra and $\Grr{0}{\gamma}{S}=\Z$. By~\ref{lem:K0-lambda-functor} and~\ref{pro:gamma-filtration}, the construction gives a contravariant functor from connected schemes $S$ to $\Z-$augmented graded rings. Also defining for each $q\in \Z$ the $K_0(S)$-module $\Grr{q}{\rm top}{S}$ as the quotient $\Fil{q}{\rm top}{S}/\Fil{q+1}{\rm top}{S}$ we obtain the graded $K_0(S)-$module
\[
\Grr{\bullet}{\rm top}{S}:=\coprod_{q\geq 0} \Grr{q}{\rm top}{S}.
\]
By~\ref{thm:gamma-top-filtration} this is a graded $\Grr{\bullet}{\gamma}{S}-$module with $\Grr{0}{\rm top}{S}=\Z$ for connected $S$. 

It is an interesting question to determine whether for a noetherian scheme $S$ (admitting an ample invertible sheaf) the $\Grr{\bullet}{\gamma}{S}-$module $\Grr{\bullet}{\rm top}{S}$ is finitely generated. 
\end{remark}
\section{Natural operations and coniveau filtration}\label{sec:natural-operations-k0}
For a regular noetherian separated scheme $S$, again we consider the isomorphism $\uptheta_S\colon K_0(S)\to K_0'(S)$ from~\ref{pro:theta_X-iso}. This is used to define a $\lambda-$ring structure on $K_0'(S)$: i.e. the collection of $\uptheta_S\circ \lambda^n\circ \uptheta_S^{-1}$ (denoted again by $\lambda^n$) for $n\geq 0$ defines a $\lambda-$ring structure on $K_0'(S)$.
\begin{proposition}\label{prop:fil-top-lambda-ideal}
Let $S$ be a regular noetherian separated scheme. For each $q\geq 0$, the subgroup $\vartheta^{-1}\Fil{q}{\rm top}{S}$ of $K_0(S)$ is a $\lambda-$ideal.
\end{proposition}
\begin{proof}
Let $\cali{E}$ be a coherent $S-$module whose support has codimension $\geq q$. There is a closed subscheme $i\colon Z\to S$ of codimension $\geq q$ such that $e:=\class{}{\cali{E}}$ is in the image of $i_\ast\colon K_0'(Z)\to K_0'(S)$; use (the proof of) [EGA III$_1$, 3.1.2]. Let $j\colon U\to S$ be the open complement of $Z$. Let $n\geq 1$. The homomorphism $u^\ast\colon K_0(S)\to K_0(U)$ is by~\ref{lem:K0-lambda-functor} a $\lambda-$morphism and hence $u^\ast\lambda^n(e)=\lambda^n(u^\ast e)=0$. Using~\ref{thm:G0-localization} (and the discussion above), it follows that $\lambda^n(e)$ belongs to the image of $i_\ast$; i.e. $\lambda^n(e)\in \Filt{q}{\rm top}$. 
\end{proof}
From now on and for $S$ as in~\ref{prop:fil-top-lambda-ideal} we consider $K_0'(S)$ as a $\lambda-$ring without mentioning $\uptheta_S$. 

For any $\lambda$-ring $A$, there are defined Adams operations $\psi_n\colon A\to A$ for $n\geq 1$;
\begin{equation}\label{def:adams}
-t\frac{d\lambda (x)}{dt}/\lambda(x)=\sum_{n\geq 1} \psi_n(x)(-t)^n.
\end{equation}
Equivalently, we let $\psi_1=\id{A}$ and for any $n\geq 1$ and $x\in A$, the element $\psi_n(x)$ is defined inductively by the equation
\begin{equation*}
-n\lambda^n(x)=(-1)^n\psi_n(x)+(-1)^{n-1}\psi_{n-1}(x)\lambda^1(x)+\dotsc+(-1)\psi_1(x)\lambda^{n-1}(x).
\end{equation*}
\begin{cor}\label{cor:adams-nlambda}
For any $q\geq 1$, $n\geq 1$ and $x\in \Fil{q}{\rm top}{S}$ we have
\[
\psi_n(x)+(-1)^nn\lambda^n(x)\in \Fil{q+1}{\rm top}{S}.
\]
\end{cor}
\begin{proof}
First note that by~\ref{prop:fil-top-lambda-ideal} and the definition of Adams operations, the element $\psi_i(x)$ belongs to $\Filt{q}{\rm top}$ for each $i\geq 1$. But $q\geq 1$ and hence $\psi_i(x)\in \Filt{1}{\gamma}$. This means that $\psi_i(x)\lambda^{n-i}(x)$ belongs to $\Filt{q+1}{\rm top}$ for $1\leq i\leq n-1$. This together with the equation defining $\psi_n$ prove the assertion.  
\end{proof}
To prove our next result, we need to recall one more notation from [SGA 6, Exp. V, 4.9]. For $A$ as above fix an element $N$ such that $\lambda^k(N)=0$ for all $k>d$. Let $n\geq 1$ and $x\in A$. The element $\lambda^n(x\lambda_{-1}(N))$ is divisible by $\lambda_{-1}(N)$. There are (universally) well-defined elements $\lambda^n(N, x)$ such that
\[
\lambda^n(x\lambda_{-1}(N))=\lambda^n(N, x)\lambda_{-1}(N).
\]
Similarly $\gamma^n(N, x)$ is defined.
\begin{lemma}\label{lem:SGA6-V-6.10}
Let $R$ be a $\Z-$augmented $\lambda-$ring, $N$ an element of $R$ such that $\lambda^k(N)=0$ for $k>d$ and $x\in \Filt{q}{\gamma}$ with $q\geq 0$. Then for each $n\geq 0$
\[
\lambda^n(N,x)+(-1)^nn^{q+d-1}x\in \Filt{q+1}{\gamma}.
\]
\end{lemma}
\begin{proof}
We shall follow the setting and notations in the proof of [SGA 6, Exp. V, 6.10]. We may write
\[
x=\sum_\alpha a_\alpha \gamma^{i_1}(x_{\alpha,1})\dotsc  \gamma^{i_k}(x_{\alpha,k})
\]
where $k\geq 0$, $a_\alpha\in \Z$, $i_1+\dotsc i_k\geq q$ and $x_{\alpha,j}\in \Filt{1}{\gamma}$. We consider the universal setting for the assertion; let $R_0$ be the (universal) $\lambda-$ring over $\Z$ generated by $N'_0$ and $x_{\alpha,j}$ subject to the relation $\gamma^k(N'_0)=0$ for all $k>d$. The augmentation $\epsilon\colon R_0\to \Z$ is defined to vanish on the generators. Let $x_0\in R_0$ be defined by the sum appearing in the equation above for $x$. Set $y=(-1)^dx_0\gamma^d(N'_0)$. Note that $y\in \Filt{q+d}{\gamma}$. In the $\lambda-$ring $R_0$ we have
\[
\lambda^n(y)+(-1)^nn^{q+d-1}y\in \Filt{q+d+1}{\gamma}.
\]
Set $N_0:=N'_0-d$. By definition $\lambda^n(y)=\lambda^n(N_0,x)\lambda_{-1}(N_0)$ and $\lambda_{-1}(N_0)=(-1)^d\gamma^d(N'_0)$. Therefore
\[
(-1)^d\gamma^d(N'_0)\bigl(\lambda^n(N_0, x_0)+(-1)^nn^{q+d-1}x_0\bigr)\in \Filt{q+d+1}{\gamma}
\]
Since the ring $R_0$ is, by~[SGA 6, Exp. V, 4.9.1], a polynomial ring in $\gamma^m(N'_0)$ for $1\leq m\leq d$ and $\gamma^p(x_{\alpha,j})$ for $p\geq 1$ and its $\gamma-$filtration coincides with that given by the degree of the latter generators, it follows that
\[
\lambda^n(N_0, x_0)+(-1)^nn^{q+d-1}x_0\in \Filt{q+1}{\gamma}.
\]
To prove the assertion, it is enough to apply to this the $\lambda-$morphism $R_0\to R$ given by $N_0\mapsto N$ and $x_{\alpha,j}\mapsto x_{\alpha, j}$.
\end{proof}
\begin{theorem}\label{thm:adams-top-fil}
If $S$ is a regular noetherian separated scheme, then for any $n\geq 1$, $q\geq 1$ and $x\in \Fil{q}{\rm top}{S}$ we have
\[
\psi_n(x)-n^{q}x\in \Filt{q+1}{\rm top}
\]
\end{theorem}
\begin{proof}
First note that by~\ref{cor:adams-nlambda} (or its proof), the ideals $\Filt{r}{\rm top}$ are stable under the operations $\psi_n$. In view of~\ref{lem:k0-cycle-length}, this shows that it is enough to prove that if $i\colon Z\to S$ is an integral closed subscheme of codimension $q$, and $x\in K_0'(Z)$, then $\alpha:=\psi_n(i_\ast x)-n^qi_\ast(x)$ belongs to $\Filt{q+1}{\rm top}$ (in fact we just need the case $x=\class{}{\cali{O}_Z}$). We may assume that $i$ is a regular immersion. To see this note that since $S$ is regular, the irreducible subscheme $Z$ is generically regularly embedded in $S$; there is an open subscheme $j\colon U\to S$ containing the generic point of $Z$ with $S\setminus U\subset Z$ of codimension $\geq 1$ and such that $Z\cap S\to U$ is regular of codimension $q$. Using the localization exact sequence
\[
K_0'(S\setminus U)\xrightarrow{} K_0'(S)\xrightarrow{j^\ast} K_0'(U)\to 0
\]
and~\ref{lem:j*-surjective}, it is enough to show $j^\ast(\alpha)\in \Filt{q+1}{\rm top}$. For the latter note that as $j^\ast$ is a $\lambda-$morphism (cf.~\ref{lem:K0-lambda-functor} and the discussions following and preceding~\ref{prop:fil-top-lambda-ideal}), we can replace $i$ by $Z\cap S\to U$. Therefore assume that $i\colon Z\to S$ is regular. Using~\ref{cor:adams-nlambda} and the fact that $Z$ is regular (or restricting to the case $x\in {\rm im}(\uptheta_Z)$), we are reduced to show
\[
\lambda^n(i_\ast x)+(-1)^{n}n^{q-1}i_\ast (x)\in \Filt{q+1}{\rm top}
\]
for all $x\in K_0(Z)$. The result~\cite[2.1]{jouanolou-1970} states that $\lambda^n(i_\ast x)=i_\ast \lambda^n(N, x)$ in $K_0(S)$ where $N$ is the class of the co-normal sheaf ${\cali{N}_{Z/S}}$ of $Z\hookrightarrow S$. Use~\ref{lem:SGA6-V-6.10} for $(N, d, q,n)=(\class{}{\cali{N}_{Z/S}},q, 0,n)$ to obtain
\[
\lambda^n(N, x)+(-1)^nn^{q-1}x\in \Filt{1}{\gamma}
\]
in $K_0(Z)$. Apply $i_\ast$ to the above and note that $i_\ast \Filt{1}{\gamma}\subseteq \Filt{q+1}{\rm top}$. The proof of~\ref{thm:adams-top-fil} is thus complete.  
\end{proof}
\begin{remark}
The result~\ref{thm:adams-top-fil} can be stated and proved in a slightly different form. Namely, let $i\colon Y\hookrightarrow X$ be a regular closed immersion of noetherian schemes. Then for any $y\in K_0(Y)$ and any $n\geq 1$ we have
\begin{equation}\label{eq:jouanolou-2.1-psi}
\psi_n(i_\ast y)=i_\ast \psi_n(N, y)
\end{equation}
where $N\in K_0(Y)$ is the class of the co-normal sheaf of $Y\hookrightarrow X$ and $\psi_n(N, x)$ is defined as follows: let $A$, $N$ and $x$ be as in the discussion preceding~\ref{thm:adams-top-fil}. Define $\psi_n(N, x)$ by the generating function
\begin{equation*}\label{def:adams-N}
-t\frac{d\lambda (N, x)}{dt}/\lambda(x)=\sum_{n\geq 1} \psi_n(N, x)(-t)^n
\end{equation*}
where $\lambda (N, x)=\sum_{n\geq 0} \lambda^n(N, x)t^n$. Proofs and details are left to reader.
\end{remark}
Now we consider the operations $\gamma^i$. Let $x\in \Filt{q}{\rm top}$. It follows from the definition~\eqref{eq:gamma-series} of $\gamma^n$'s and~\ref{prop:fil-top-lambda-ideal} that $\gamma^n(x)\in \Filt{q}{\rm top}$ for all $n\geq 1$. Also note that if $x\in \Filt{q}{\rm top}$ with $q\geq 1$, then for each $n>q$ we have $\gamma^n(x)\in \Filt{q+1}{\rm top}$. The essential property is however the following.	
\begin{proposition}\label{prop:gamma-eigenvalue-top}
Let $S$ be a regular noetherian separated scheme. If $x\in \Filt{q}{\rm top}$ with $q\geq 1$, then
\[
\gamma^q(x)-(-1)^{q-1}(q-1)!x\in \Filt{q+1}{\rm top}
\]
\end{proposition}
\begin{proof}
The proof is similar to~\ref{thm:adams-top-fil}. 
\end{proof}
\begin{cor}\label{cor:top-gamma-torsion}
Let $S$ be a regular noetherian separated scheme of finite dimension $d$. For any integer $q\geq 1$ we have
\[
(d-1)!(d-2)!\dotsc (q-1)!\Fil{q}{\rm top}{S}\subseteq \Fil{q}{\gamma}{S}.
\]
\end{cor}
\begin{proof}
Let $x\in \Filt{q}{\rm top}$. By~\ref{prop:gamma-eigenvalue-top} we have $x_1:=(q-1)!x\in \Filt{q+1}{\rm top}+\Filt{q}{\gamma}$. Using~\ref{prop:gamma-eigenvalue-top} again (for $q+1$) we obtain $x_2:=q!x_1\in \Filt{q+2}{\rm top}+\Filt{q}{\gamma}$. Using this inductively and noting that $\Filt{d+1}{\rm top}=0$ we obtain the result. (Compare this with [SGA 6, Exp. VII, 4.11]). 
\end{proof}
It follows from~\ref{thm:gamma-top-filtration} that for each integer $q\geq 0$ there is defined a natural map
\[
\uprho\colon \Grr{q}{\gamma}{S}\rightarrow\Grr{q}{\rm top}{S}.
\]
\begin{cor}\label{thm:gamma-top-iso-Q}
Let $S$ be a regular noetherian separated scheme of finite dimension. The natural map
\[
\uprho\otimes\id{\Q}\colon \Grr{q}{\gamma}{S}\otimes\Q\rightarrow\Grr{q}{\rm top}{S}\otimes\Q
\]
is an isomorphism for any $q\geq 0$.
\end{cor}
\begin{proof}
This follows directly from~\ref{cor:top-gamma-torsion}.
\end{proof}

\end{document}